\documentclass[11pt, reqno]{amsart}
\usepackage[margin=3.5cm]{geometry}
\usepackage{mathabx}
\usepackage{cancel}
\usepackage{some_macros}

\let\oldtocsection=\tocsection
\let\oldtocsubsection=\tocsubsection

\newcommand{\treesor}{\mathfrak{T}^\op{or}}
\newcommand{\treesun}{\mathfrak{T}^\op{un}}
\newcommand{\Vin}{V_\op{in}}
\newcommand{\Vmov}{V_\op{mov}}
\newcommand{\Vleaf}{V_\op{leaf}}

\newcommand{\aut}{\op{Aut}}
\newcommand{\orb}{\mathfrak{o}}
\newcommand{\aug}{\epsilon}

\newcommand{\shuf}{\op{Sh}}
\newcommand{\shufbar}{\ovl{\shuf}}

\newcommand{\veca}{{\vec{a}}}

\newcommand{\mc}{\mathfrak{m}}

\newcommand{\setm}{\;\setminus\;}

\newcommand{\Ga}{\Gamma}

\newcommand{\dg}{\frak{q}}
\newcommand{\pt}{{pt}}
\newcommand{\mult}{\op{mult}}

\newcommand{\leaf}{\ell}
\newcommand{\calS}{\mathcal{S}}
\newcommand{\Tcount}{\mathbf{T}}
\newcommand{\wtTcount}{\wt{\Tcount}}

\newcommand{\intE}{{\mathring{E}}}

\renewcommand{\tocsection}[2]{\hspace{0em}\oldtocsection{#1}{#2}}
\renewcommand{\tocsubsection}[2]{\hspace{1em}\oldtocsubsection{#1}{#2}}

\tikzset{node distance=3cm, auto}

\makeatletter
\def\@secnumfont{\bfseries}
\def\section{\@startsection{section}{1}%
  \z@{.7\linespacing\@plus\linespacing}{.5\linespacing}%
  {\normalfont\Large\bfseries}}
\def\acksection{\@startsection{subsubsection}{1}%
  \z@{.7\linespacing\@plus\linespacing}{.5\linespacing}%
  {\normalfont\Large\bfseries}}
\def\subsection{\@startsection{subsection}{2}%
  \z@{.75\linespacing\@plus.7\linespacing}{-.5em}%
  {\normalfont\large\bfseries}}
\def\subsubsection{\@startsection{subsubsection}{3}%
  \z@{.75\linespacing\@plus.7\linespacing}{-.5em}%
  {\normalfont\bfseries}}
\makeatother

\newtheorem{thm}{Theorem}[section]
\newtheorem{thmlet}{Theorem}

\newtheorem{prop}[thm]{Proposition}

\newtheorem{corlet}[thmlet]{Corollary}

\newtheorem{notation}[thm]{Notation}

\newtheorem{conjecture}[thm]{Conjecture}
\newtheorem{definition}[thm]{Definition}

\theoremstyle{remark}
\numberwithin{equation}{section} 

\newtheoremstyle{customremark}
{8pt}
{8pt}
{}
{}
{\bfseries}
{.}
{.5em}
{}
\theoremstyle{customremark}
\newtheorem{rmk_no_diamond}[thm]{Remark}
\newenvironment{rmk}{\begin{rmk_no_diamond} } {\hfill$\Diamond$ \end{rmk_no_diamond}}
\newtheorem{example_no_diamond}[thm]{Example}
\newenvironment{example}{\begin{example_no_diamond} } {\hfill$\Diamond$ \end{example_no_diamond}}

\begin{document}

\makeatother

\date{\today}

\title{A tree formula for the ellipsoidal superpotential of the complex projective plane}
\date{\today}
\pagestyle{plain}

\begin{abstract}
The ellipsoidal superpotential of the complex projective plane can be interpreted as a count of rigid rational plane curves of a given degree with one prescribed cusp singularity.
In this note we present a closed formula for these counts as a sum over trees with certain explicit weights.  This is a step towards understanding the combinatorial underpinnings of the ellipsoidal superpotential and its mysterious nonvanishing and nondecreasing properties.
\end{abstract}

\author{Kyler Siegel}
\thanks{K.S. is partially supported by NSF grant DMS-2105578.}

\maketitle

\tableofcontents

\section{Introduction}

Given a closed symplectic manifold $M^{2n}$, a homology class $A \in H_2(M)$, and a tuple of positive real numbers $\veca = (a_1,\dots,a_n) \in \R_{>0}^n$, the {\bf ellipsoidal superpotential} $\Tcount_{M,A}^\veca \in \Q$ is an enumerative invariant which encodes important information about (a) stabilized symplectic embeddings of ellipsoids into $M$ and (b) singular rational curves in $M$.
In this note we focus on the case of the complex projective plane $M = \CP^2$, and we put $\Tcount_{d}^a := \Tcount_{\CP^2,d[L]}^{(1,a)}$ for $a \in \R_{> 0}$\footnote{In the definition it is convenient to assume that $a$ is irrational, and we extend this to rational $a \in \R_{>0}$ by the convention $\Tcount_d^a := \Tcount_d^{a + \delta}$ for $\delta > 0$ sufficiently small.} and $d \in \Z_{\geq 1}$, where $[L] \in H_2(\CP^2)$ is the line class.

For instance, when $\Tcount_{d}^a \neq 0$ we get an obstruction to symplectic embeddings of the form $E(\mu,\mu a) \times \C^N \hooksymp \CP^2 \times \C^N$ for $\mu \in \R_{>0}$ and $N \in \Z_{\geq 0}$, and it is expected that together these give a complete set of stable obstructions for symplectic embeddings of ellipsoids into $\CP^2$ (c.f. \cite[\S2.7]{cusps_and_ellipsoids}).
Moreover, when $a = p/q$ is a reduced fraction such that $p+q =3d$, we have $\Tcount_d^{p/q} \neq 0$ if and only if there exists a genus zero degree $d$ singular symplectic curve in $\CP^2$ which has one $(p,q)$ cusp and is otherwise positively immersed (see \cite[Thm. E]{cusps_and_ellipsoids}).
We call such singular curves {\bf sesquicuspidal}, and their existence in both the algebraic and symplectic categories are subtle problems which are closely linked.

To define $\Tcount_{M,A}^\veca$, we first consider the compact symplectic manifold with boundary $M_\veca := M \setm \iota(\intE(\eps\veca))$, where $\iota: E(\eps\veca) \hooksymp M$ is a symplectic embedding for some small $\eps > 0$.
Here $E(\veca) = \{ \pi\sum_{i=1}^n \tfrac{1}{a_i}|z_i|^2 \leq 1\} \subset \C^n$ denotes the closed symplectic ellipsoid with area factors $(a_1,\dots,a_n)$ and $\intE(\veca)$ denotes its interior.
Let $\wh{M}_\veca := M_\veca \cup (\R_{\leq 0} \times \bdy M_\veca)$ denote the symplectic completion of $M_\veca$, and let $J$ be an admissible almost complex structure on $\wh{M}_\veca$ in the sense of symplectic field theory (SFT).
Then $\Tcount_{M,A}^\veca$ is the count of index zero finite energy $J$-holomorphic planes $u: \C \ra \wh{M}_\veca$.
In general this is a virtual count taking rational values, although it is shown in \cite[\S2]{cusps_and_ellipsoids} that in many important cases it can be defined using classical pseudoholomorphic curve techniques and takes integer values. For instance, this is the case for $\Tcount_d^{p/q}$ whenever $p/q$ is a reduced fraction satisfying $p+q=3d$.

As illustrated above, a central question is to understand when $\Tcount_{M,A}^\veca$ is nonzero.
For example, we have:
\begin{conjecture}\label{conj:p_q_d}
 For any reduced fraction $p/q > 1$ satisfying $p+q = 3d$ and $(p-1)(q-1) \leq (d-1)(d-2)$, we have $\Tcount_{d}^{p/q} \neq 0$.
\end{conjecture}
\NI This is equivalent to the statement that there exists an index zero $(p,q)$-sesquicuspidal symplectic curve in $\CP^2$ of genus zero and degree $d$ if and only if it is allowed by the adjunction formula, which states here
 that the count of singularities excluding the $(p,q)$ cusp is $\tfrac{1}{2}(d-1)(d-2) - \tfrac{1}{2}(p-1)(q-1)$.
It is known that an affirmative answer would in particular imply optimality of Hind's folding embedding $E(\mu,\mu a) \times \C^N \hooksymp \CP^2 \times \C^N$ with $\mu = \tfrac{a+1}{3a}$ for all $N \in \Z_{\geq 1}$ and $a > \tau^4$, where $\tau = \tfrac{1+\sqrt{5}}{2}$ is the golden ratio (see \cite{hind2015some}).
Another closely related question concerns the behavior of $\Tcount_d^a$ as a function of $a$:
\begin{conjecture}\label{conj:nondec_in_a}
  The count $\Tcount_{d}^a \in \R$ is nondecreasing as a function of $a \in \R_{>1}$. 
\end{conjecture}

The recent article \cite{SDEP} gives a recursive formula for $\Tcount_{M,A}^\veca$, which takes the following form in the case $M = \CP^2$.
Firstly, we associate to each $a \in \R_{>0}$ a unit step lattice path $\Ga_0^a,\Ga_1^a,\Ga_2^a,\dots \in \Z_{\geq 0}^2$. Explicitly, for $k \in \Z_{\geq 0}$ and $a$ irrational, $\Ga^a_k$ is the pair $(i,j) \in \Z_{\geq 0}^2$ which minimizes $\max \{i,aj\}$ subject to $i+j = k$.
For example, in the case $a = \tfrac{3}{2}+\delta$ with $\delta > 0$ sufficiently small the first few terms are $(0,0),(1,0),(1,1),(2,1),(3,1),(3,2),(4,2),(4,3)$, and so on (see \cite[Fig. 1]{SDEP} for an illustration).

\begin{thm}[\cite{SDEP}]\label{thm:rec}
For any $a \in \R_{> 0}$ and $d \in \Z_{\geq 1}$ we have:
\begin{align}\label{eq:wtTcount_intro}
\wtTcount_d^a = \left(\Ga^a_{3d-1}\right)!\left( (d!)^{-3} - \sum_{\substack{k \geq 2\\d_1,\dots,d_k \in \Z_{\geq 1}\\d_1 + \cdots + d_k = d}}  \frac{\wtTcount_{d_1}^a \cdot \cdots \cdot \wtTcount_{d_k}^a}{k!(\sum_{s=1}^k \Ga^a_{3d_i-1} )!}\right).
\end{align} 
\end{thm}

Here we put $\wtTcount_d^a := \mult_a(\Ga^a_k)\cdot \Tcount_d^a$, where $\mult_a(i,j) = i$ if $i > aj$ and $\mult_a(i,j) = j$ otherwise, and we write \eqref{eq:wtTcount_intro} using $\wt\Tcount_d^a$ instead of $\Tcount_d^a$ as convenience which yields a slightly simpler formula.
We add pairs in $\Z_{\geq 0}^2$ componentwise in the usual fashion, and we define the factorial of a pair $(i,j)$ by $(i,j)! := i!j!$.
Note that all dependence on $a$ in \eqref{eq:wtTcount_intro} is via the lattice path $\Ga_0^a,\Ga_1^a,\Ga_2^a,\dots \in \Z_{\geq 0}^2$.
The term $(d!)^{-3}$ arises from the computation of degree zero genus zero stationary gravitational descendant Gromov--Witten invariants of $\CP^2$ (see \cite[\S5.3.2]{SDEP} and \eqref{eq:CP2_stat_des} below).

Although Theorem~\ref{thm:rec} and its generalization makes it possible to compute $\Tcount_d^a$ for any fixed $d \in \Z_{\geq 1}$ and $a \in \R_{>0}$ given enough computational power, its recursivity somewhat obscures its enumerative essence, and the presence of terms of both positive and negative sign complicates efforts to study e.g. Conjecture~\ref{conj:p_q_d} or Conjecture~\ref{conj:nondec_in_a}.
Even more basically, although it is known by geometric arguments \cite{Mint,McDuffSiegel_counting} that we have $\Tcount_d^\infty > 0$ for all $d \in \Z_{\geq 1}$\footnote{Here we put $\Tcount_d^\infty := \Tcount_d^a$ for $a \gg d$. This corresponds to the count from \cite{McDuffSiegel_counting} of degree $d$ rational plane curves satisfying an order $3d-1$ local tangency constraint, which essentially amounts to fixing the $(3d-2)$-jet at a point.}, this is not a priori clear from Theorem~\ref{thm:rec}.\footnote{A fortiori, it follows by the obstruction bundle gluing method of \cite{Mint} that $\Tcount_d^\infty$ is nondecreasing as a function of $d \in \Z_{\geq 1}$.}
Similarly, it follows by automatic transversality and the results in \cite[\S2]{cusps_and_ellipsoids} that $\Tcount_d^{p/q}$ is a nonnegative integer whenever $p/q > 1$ is a reduced fraction with $p+q = 3d$, but this is not obvious from \eqref{eq:wtTcount_intro} due to the presence of denominators.

\sss

The above considerations motivate the search for a positive combinatorial formula for $\Tcount_d^a$.
In this note we take a step in this direction by establishing a closed formula for $\Tcount_d^a$ as a sum over trees with $d$ leaves and certain explicit, combinatorially defined weights.
Our formula still has some negative terms which we are not able to avoid, but they enter in a fairly transparent way, which could open further avenues for studying $\Tcount_{M,A}^\veca$ via combinatorics.

Before stating the formula, we will need some graph theoretic terminology.
\begin{definition}
For $k \in \Z_{\geq 1}$, let $\treesun_k$ denote the set of (isomorphism classes of) rooted trees with $k$ {\em unordered} leaves and no bivalent vertices (i.e. no vertices with $|v| = 2$).   
\end{definition}
\NI In other words, a tree $T \in \treesun_k$ has a distinguished root vertex and $k$ leaf vertices, and the remaining vertices are called {\bf internal}.
We denote the set of leaf vertices by $\Vleaf(T)$ and the set of internal vertices by $\Vin(T)$.
See Figure~\ref{fig:T_4_un} for a picture of $\treesun_4$.

We orient all edges of $T$ towards the root, and we will say that $v$ is ``above'' $w$ if $w$ lies on the oriented path from $v$ to the root.
\begin{definition}
  For a tree $T$ in $\treesun_k$, an internal vertex $v \in \Vin(T)$ is {\bf movable} if there are no internal vertices above it.  We denote the set of movable vertices of $T$ by $\Vmov(T) \subset \Vin(T)$. 
\end{definition}

\begin{definition}
For a vertex $v$ of a tree $T$ in  $\treesun_k$, the {\bf leaf number} $\leaf(v)$ is the number of leaf vertices lying above $v$, including possibly $v$ itself (i.e. $\leaf(v) = 1$ if $v$ is a leaf vertex).
\end{definition}

\begin{figure}  
  \includegraphics[scale=1]{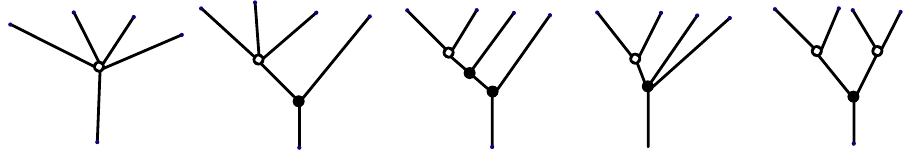}
  \caption{The trees comprising $\treesun_4$. The top four vertices are the leaves and the bottom vertex is the root. The internal vertices are denoted with a large circle, which is open for a movable vertex and solid otherwise.
  The values of $|\aut(T)|$ are $24,6,2,4,8$ respectively.}
  \label{fig:T_4_un}
\end{figure}

We are now ready to state our main result:
\begin{thmlet}\label{thmlet:main_tree}
For any $d \in \Z_{\geq 1}$, we have:
\begin{align}\label{eq:main_tree_formula}
\wtTcount_d^a = (\Ga^a_2!)^d\sum_{T \in \treesun_d}\frac{(-1)^{|\Vin(T) \setm \Vmov(T)|}}{|\aut(T)|} \cdot \prod\limits_{v \in \Vin(T)}\left(\frac{\Ga^a_{3\leaf(v)-1}!}{\left(\sum\limits_{v' \ra v} \Ga^a_{3\leaf(v')-1}\right)!}\right)\cdot \prod\limits_{v \in \Vmov(T)} \left((\leaf(v)!)^{-2} \frac{ \left( \leaf(v)\Ga^a_{2} \right)!}{(\Ga^a_2!)^{\leaf(v)}}-1\right).
\end{align}
\end{thmlet}
\NI Here the sum $\sum\limits_{v' \ra v}$ is over all vertices $v'$ which are adjacent to $v$ and lie strictly above it.
\begin{rmk}
Note that we have $\Ga^a_2 = 1$ for $1 < a < 2$ and $\Ga^a_2 = 2$ for $a \geq 2$.
Therefore the last parenthesized term in \eqref{eq:main_tree_formula} is 
$\frac{1}{\leaf(v)!}-1$ if $1 < a < 2$ and $2^{-\leaf(v)}\binom{2\leaf(v)}{\leaf(v)}-1$ if $a \geq 2$.
In the latter case one can check that the term $2^{-\leaf(v)}\binom{2\leaf(v)}{\leaf(v)}-1$ is always positive, so a summand in \eqref{eq:main_tree_formula} is negative if and only if there are an odd number of ``unmovable'' vertices $v \in \Vin(T) \setm \Vmov(T)$.
\end{rmk}

In the special case $a \gg 1$, we have $\Ga^\infty_k = (k,0)$ for all $k \in \Z_{\geq 0}$, so we get the following formula for $\wtTcount_d^\infty = (3d-1)\Tcount_d^\infty$:
\begin{corlet}\label{corlet:wtT_d_rec}
For $d \in \Z_{\geq 1}$ we have:
\begin{align}\label{eq:wtT}
\wtTcount_d^\infty = 2^d\sum_{T \in \treesun_d}\frac{(-1)^{|\Vin(T) \setm \Vmov(T)|}}{|\aut(T)|} \cdot \prod\limits_{v \in \Vin(T)}\left(\frac{(3\leaf(v)-1)!}{\left(3\ell(v) - |v| +1\right)!}\right)\cdot \prod\limits_{v \in \Vmov(T)} \left(2^{-\leaf(v)}\binom{2\leaf(v)}{\leaf(v)}-1\right).
\end{align}
\end{corlet}
\NI Here $|v|$ denotes the valency (a.k.a. degree) of $v$. Equivalently, this is the number of incoming edges plus $1$, where 
for $v \in \Vin(T)$ the {\bf incoming} (resp. {\bf outgoing}) edges of $v$ are those edges of $T$ which have $v$ as an endpoint and are oriented towards (resp. away from) $v$.

\begin{example}
In the case $d=1$, there is just a single tree in $\treesun_1$, which has no internal vertices. Both products in \eqref{eq:wtT} are vacuous, so we get $\wtTcount_1^\infty = 2$ and thus $\Tcount_1^\infty = \tfrac{1}{2}\wtTcount_1^\infty = 1$.
\end{example}

\begin{example}
In the case $d=2$, there is also just one tree in $\treesun_2$, which has one internal vertex, and that vertex is movable with leaf number is $2$.
Then \eqref{eq:wtT} gives
\begin{align*}
\wtTcount_2^\infty = 2^2 \cdot \frac{(-1)^0}{2}\cdot \frac{5!}{4!} \cdot \left(2^{-2}\binom{4}{2} -1\right) = 5,
\end{align*}
and thus $\Tcount_2^\infty = \tfrac{1}{5}\wtTcount_2^\infty = 1$.
\end{example}

\begin{example}
To compute $\wtTcount_3^\infty$, put $\treesun_3 = \{T_1,T_2\}$, where $T_1$ has a single internal vertex of valency $4$ and $T_2$ has two internal vertices each of valency $3$.
We have $|\aut(T_1)| = 3!$ and $|\aut(T_2)| = 2!$. 
The internal vertex of $T_1$ is movable and has leaf number $3$.
The internal vertices of $T_2$ have leaf numbers $2,3$, and the first is movable while the second is not.
Plugging these into \eqref{eq:wtT} gives 
\begin{align*}
\wtTcount_3^\infty &= 2^3\left(\frac{1}{3!}\frac{8!}{6!}\left[2^{-3}\binom{6}{3}-1\right] - \frac{1}{2!}\frac{5!}{4!}\frac{8!}{7!}\left[2^{-2}\binom{4}{2}-1\right]\right)\\
&= 2^3(14-10) = 32, 
\end{align*}
and hence
$\Tcount_3^\infty = \tfrac{1}{8}\wtTcount_3^\infty = 4$.
\end{example}

  \begin{rmk}
 Let $\treesor_d$ be defined just like $\treesun_d$, except with {\em ordered} leaves (see Definition~\ref{def:treesor}).
 Curiously, we have precisely $\Tcount_d^\infty = |\treesor_d|$ for $d = 1,2,3,4$, while experimentally we have $\Tcount_d^\infty < |\treesor_d|$ for $d \geq 5$.
 This perhaps suggests that $\Tcount_d^\infty$ counts elements of $\treesor_d$ with some additional conditions which only become relevant when there are at least $5$ leaves.
\end{rmk}

The main ingredients in the proof of Theorem~\ref{thmlet:main_tree} are (i) the computation from \cite{SDEP} of genus zero punctured stationary descendants in symplectic ellipsoids (see Theorem~\ref{thm:punc_desc} below), and (ii) homological perturbation theory for $\Li$ algebras.
First, in \S\ref{sec:HPT} we recall some relevant formalism for $\Li$ algebras and state an explicit formula for inverting $\Li$ homomorphisms between evenly graded $\Li$ algebras. Then, in \S\ref{sec:tree_formula} we combine this with the computation for punctured stationary descendants in ellipsoids to give a closed formula for $\Tcount_d^a$, which after some manipulations yields \eqref{eq:main_tree_formula}.

\acksection*{Acknowledgements}
{
This note is an offshoot of the joint project \cite{SDEP} with Grisha Mikhalkin, to whom I am grateful for many helpful discussions.
}

\section{Homological perturbation theory}\label{sec:HPT}

$\Li$ algebras provide a useful framework for organizing the curve counts underlying the ellipsoidal superpotential.
Here we will give a minimal treatment, referring the reader to \cite[\S3]{SDEP} and the references therein for more details.
An $\Li$ algebra over $\Q$ by definition consists of a $\Z$-graded rational vector space $V$ along multilinear $k$-to-$1$ maps $\ell^k$ for $k \in \Z_{\geq 1}$ which are symmetric (in a suitable graded sense) and satisfy an infinite sequence of quadratic relations.
As it happens, in this note we will only need to consider evenly graded $\Li$ algebras over $\Q$, which are simply evenly graded rational vector spaces.
Indeed, if $V$ is only supported in even degrees then all of the $\Li$ operations $\ell^1,\ell^2,\ell^3,\dots$ necessarily vanish for degree parity reasons, and moreover all of the Koszul-type signs conveniently disappear.

Given an evenly graded $\Li$ algebra $V$, its (reduced) symmetric tensor coalgebra is 
$\ovl{S}V = \oplus_{k=1}^\infty \odot^kV$, where $\odot^kV := (\underbrace{V \otimes \cdots \otimes V}_k)/\Sigma_k$ is the $k$-fold symmetric tensor power of $V$.
Here $\Sigma_k$ is the permutation group on $k$ elements, acting on $\odot^kV$ in the obvious way (without signs).
The coproduct on $\Delta_{\ovl{S}V}: \ovl{S}V \ra \ovl{S}V \otimes \ovl{S}V$ is given by
\begin{align*}
\Delta_{\ovl{S}V}(v_1 \odot \cdots \odot v_k) = \sum_{i=1}^{k-1} \sum_{\sigma \in \shuf(i,k-i)} (v_{\sigma(1)} \odot \cdots \odot v_{\sigma(i)}) \otimes (v_{\sigma(i+1)} \odot \dots \odot v_{\sigma(k)}),
\end{align*}
where $\shuf(i,k-i)$ is the subset of permutations $\sigma \in \Sigma_k$ satisfying $\sigma(1) < \cdots < \sigma(i)$ and $\sigma(i+1) < \cdots < \sigma(k)$.

Given evenly graded $\Li$ algebras $V,W$ over $\Q$, an $\Li$ homomorphism $\Phi: V \ra W$ is simply a sequence of degree zero\footnote{Note that some references use different grading conventions.} linear maps $\Phi^k: \odot^k V \ra W$ for $k \in \Z_{\geq 1}$, i.e. each $\Phi^k$ is a multilinear map with $k$ symmetric inputs in $V$ and one output in $W$.
The maps $\Phi^1,\Phi^2,\Phi^3,\dots$ can be uniquely assembled into a degree zero coalgebra map $\wh{\Phi}: \ovl{S}V \ra \ovl{S}W$, i.e. 
a linear map satisfying 
\begin{align*}
(\wh{\Phi} \otimes \wh{\Phi})\circ \Delta_{\ovl{S}V} = \Delta_{\ovl{S}W}\circ \wh{\Phi},
\end{align*}
where
\begin{align*}
\wh{\Phi}(v_1,\dots,v_k) = \sum_{\substack{s \geq 1\\1 \leq k_1 \leq \cdots \leq k_s \\ k_1 + \cdots + k_s = k}}\sum_{\sigma \in \shufbar(k_1,\dots,k_s)} (\Phi^{k_1} \odot ... \odot \Phi^{k_s})(v_{\sigma(1)}\odot ... \odot v_{\sigma(k)}).
\end{align*}
Conversely, given such a coalgebra map $\wh{\Phi}$ we recover $\Phi^1,\Phi^2,\Phi^3,\dots$ via
the compositions 
\[\begin{tikzcd}
\odot^k V \;\;\subset\;\; \ovl{S}V & \ovl{S}W & W
\arrow["\wh{\Phi}",from=1-1,to=1-2] 
\arrow["\pi_1",from=1-2,to=1-3],
\end{tikzcd}\]
where $\pi_1: \ovl{S}W \ra W$ is the projection to tensors of word length one.

If $V_1,V_2,V_3$ are evenly graded $\Li$ algebras over $\Q$ and we have $\Li$ homomorphisms $\Phi: V_1 \ra V_2$ and $\Psi: V_2 \ra V_3$, then the composition $\Li$ homomorphism $\Psi \circ \Phi: V_1 \ra V_3$ is characterized by $\wh{\Psi \circ \Phi} = \wh{\Psi} \circ \wh{\Phi}$.
In terms of the operations $\Phi^1,\Phi^2,\Phi^3,\dots$ and $\Psi^1,\Psi^2,\Psi^3,\dots$ this translates into $(\Psi \circ \Phi)^1(v_1) = (\Psi^1 \circ \Phi^1)(v_1)$, $(\Psi \circ \Phi)^2(v_1,v_2) = \Psi^1(\Phi^2(v_1,v_2)) + \Psi^2(\Phi^1(v_1),\Phi^1(v_2))$, and so on.
The identity $\Li$ homomorphism $\mathbb{1}: V \ra V$ corresponds to the identity map $\ovl{S}V \ra \ovl{S}V$, or equivalently $\mathbb{1}^1: V \ra V$ is the identity map and $\mathbb{1}^k \equiv 0$ for all $k \in \Z_{\geq 2}$.

\sss

Trees arise naturally in $\Li$ contexts, and it is sometimes convenient to work with trees with {\em ordered} leaves.
\begin{definition}\label{def:treesor}
For $k \in \Z_{\geq 1}$, let $\treesor_k$ denote the set of (isomorphism classes of) rooted trees with $k$ {\em ordered} leaves and no bivalent vertices. 
  \end{definition}
\NI Note that the ordering of the leaf vertices of $T \in \treesor_k$ amounts to a bijection between $\Vleaf(T)$ and $\{1,\dots,k\}$.
We will also refer to the edges connected to the leaf vertices as the {\bf leaf edges},
and the edge connected to the root vertex as the {\bf root edge}.
The set $\treesun_k$ is the quotient $\treesor_k/\Sigma_k$ by the natural symmetric group action on $\treesor_k$ which reorders the leaves, and the stabilizer of a tree $T \in \treesor_k$ is the automorphism group $\aut(T)$ of $T$.

\sss

The following proposition is proved using standard techniques from homological perturbation theory. One can formulate a more general version without any grading restrictions,
but here we state a simplified version which suffices for our purposes.
\begin{prop}\label{prop:inverting_L_inf}
 Let $V$ and $W$ be evenly graded $\Li$ algebras over $\Q$, and let $\Phi: V \ra W$ be an $\Li$ homomorphism 
such that the linear map $\Phi^1: V \ra W$ is invertible. Then there exists an $\Li$ homomorphism
$\Psi: W \ra V$ such that 
\begin{align*}
\Psi \circ \Phi = \1_{V}\;\;\;\;\;\text{and}\;\;\;\;\; \Phi \circ \Psi = \1_{W}.
\end{align*}

Moreover, $\Psi$ is given explicitly as follows. We first set $\Psi^1: W \ra V$ to be the inverse of $\Phi^1: V \ra W$.
For $k \geq 2$, $T \in \treesor_k$, and $w_1,\dots,w_k \in W$, define
$\Psi^T(w_1,\dots,w_k)$ as follows.
Start by labeling the $i$th leaf edge of $T$ by $\Psi^1(w_i)$ for $i = 1,\dots,k$.
Recursively, for each internal vertex $v \in \Vin(T)$, say with $j$ incoming edges, 
label the outgoing edge by the result after applying $\Psi^1 \circ \Phi^j$ to the corresponding labels of its incoming edges.
We define $\Psi^T(w_1,\dots,w_k)$ to be the resulting label on the root edge, and finally put
\begin{align*}
\Psi^k(w_1,\dots,w_k) = \sum_{T \in \treesor_k} (-1)^{|\Vin(T)|}\Psi^T(w_1,\dots,w_k).
\end{align*}
\end{prop}

\begin{example}
 We have:
\begin{itemize}
    \item $\Psi^2(x,y) = -\Psi^1\Phi^2(\Psi^1(x),\Psi^1(y))$ 
  \item $\Psi^3(x,y,z) = -\Psi^1\Phi^3(\Psi^1(x),\Psi^1(y),\Psi^1(z)) + \Psi^1\Phi^2(\Psi^1(x),\Psi^1\Phi^2(\Psi^1(y),\Psi^1(z))) + \Psi^1\Phi^2(\Psi^1(y),\Psi^1\Phi^2(\Psi^1(x),\Psi^1(z))) + \Psi^1\Phi^2(\Psi^1(z),\Psi^1\Phi^2(\Psi^1(x),\Psi^1(y)))$.
\end{itemize}
\end{example}

\begin{proof}
The relation $\Psi \circ \Phi = \1$ means that for $k \geq 2$ we must have $(\Psi \circ \Phi)^k(v_1,\dots,v_k) = 0$ for any $v_1,\dots,v_k \in V$, and this amounts to
\begin{align*}
\sum_{\substack{s \geq 1\\1 \leq k_1 \leq \cdots \leq k_s \\ k_1 + \cdots + k_s = k}}\sum_{\sigma \in \shufbar(k_1,\dots,k_s)} 
\Psi^s \circ (\Phi^{k_1}\odot \cdots \odot \Phi^{k_s})(v_{\sigma(1)},\dots,v_{\sigma(k)}) = 0,
\end{align*}
or equivalently
\begin{align*}
\Psi^k(\Phi^1(v_1),\dots,\Phi^1(v_k)) = -\sum_{\substack{1 \leq s \leq k-1\\1 \leq k_1 \leq \cdots \leq k_s \\ k_1 + \cdots + k_s = k}}\sum_{\sigma \in \shufbar(k_1,\dots,k_s)} 
\Psi^s \circ (\Phi^{k_1}\odot \cdots \odot \Phi^{k_s})(v_{\sigma(1)},\dots,v_{\sigma(k)}),
\end{align*}
i.e. for any $w_1,\dots,w_k \in W$ we must have
\begin{align*}
\Psi^k(w_1,\dots,w_k) = -\sum_{\substack{1 \leq s \leq k-1\\1 \leq k_1 \leq \cdots \leq k_s \\ k_1 + \cdots + k_s = k}}\sum_{\sigma \in \shufbar(k_1,\dots,k_s)} 
\Psi^s \circ (\Phi^{k_1}\odot \cdots \odot \Phi^{k_s})(\Psi^1(w_{\sigma(1)}),\dots,\Psi^1(w_{\sigma(k)})).
\end{align*}
This is easily seen to agree with the definition of $\Psi$ given in the statement of the proposition, which therefore necessarily satisfies $\Psi \circ \Phi = \1$.

As for the relation $\Phi \circ \Psi = \1$, we need 
\begin{align*}
\sum_{\substack{s \geq 1\\1 \leq k_1 \leq \cdots \leq k_s \\ k_1 + \cdots + k_s = k}}\sum_{\sigma \in \shufbar(k_1,\dots,k_s)} 
\Phi^s \circ (\Psi^{k_1}\odot \cdots \odot \Psi^{k_s})(w_{\sigma(1)},\dots,w_{\sigma(k)}) = 0
\end{align*}
for any $w_1,\dots,w_k \in W$, 
or equivalently
\begin{align*}
\Phi^1(\Psi^k(w_1,\dots,w_k)) = 
-\sum_{\substack{s \geq 2\\1 \leq k_1 \leq \cdots \leq k_s \\ k_1 + \cdots + k_s = k}}\sum_{\sigma \in \shufbar(k_1,\dots,k_s)} 
\Phi^s \circ (\Psi^{k_1}\odot \cdots \odot \Psi^{k_s})(w_{\sigma(1)},\dots,w_{\sigma(k)}),
\end{align*}
i.e. 
\begin{align*}
\Psi^k(w_1,\dots,w_k) = 
-\sum_{\substack{s \geq 2\\1 \leq k_1 \leq \cdots \leq k_s \\ k_1 + \cdots + k_s = k}}\sum_{\sigma \in \shufbar(k_1,\dots,k_s)} 
\Psi^1 \circ \Phi^s \circ (\Psi^{k_1}\odot \cdots \odot \Psi^{k_s})(w_{\sigma(1)},\dots,w_{\sigma(k)}),
\end{align*}
which is equivalent to our definition of $\Psi^k$.
\end{proof}

\section{Tree formula}\label{sec:tree_formula}

Before proving Theorem~\ref{thmlet:main_tree}, we recall the computation of punctured stationary descendants in ellipsoids from \cite{SDEP}, and explain how this relates to the $\Li$ formalism from the previous section. 
For the time being we take $M^{2n}$ to be any closed symplectic manifold and $A \in H_2(M)$ a homology class.
We associate to each $\veca \in \R_{>0}^n$ an evenly graded $\Li$ algebra $C_\veca$ with basis given by formal symbols $\orb^\veca_i$ with degree $|\orb^\veca_i| = -2-2i$ for $i \in \Z_{\geq 1}$.
We encode the ellipsoidal superpotential in terms of these $\Li$ algebras by putting 
$\mc_{M,A}^\veca := \wtTcount_{M,A}^\veca \,\orb^\veca_{c_1(A)-1}$ and
\begin{align*}
\exp_A(\mc_{M}^\veca) := \sum_{\substack{k \geq 1\\A_1,\dots,A_k \in H_2(M)\\A_1 + \cdots + A_k = A}} \tfrac{1}{k!}\mc^\veca_{M,A_1} \odot \cdots \odot \mc^\veca_{M,A_k}.
\end{align*}
This corresponds to collections of rigid pseudoholomorphic planes in $\wh{M}_\veca$ of total homology class $A$.

We denote the genus zero Gromov--Witten invariant of $M$ in homology class $A$ and carrying a maximal order stationary descendant by
$N_{M,A}\lll \psi^{c_1(A)-2}\pt\rrr \in \Q$ (see e.g. \cite{kock2001notes} or the discussion in \cite[\S5.3]{SDEP}).
In this note we will only need the computation of these invariants for $\CP^2$, namely 
\begin{align}\label{eq:CP2_stat_des}
N_{\CP^2,d[L]}\lll \psi^{3d-2}\pt\rrr = (d!)^{-3}
\end{align}
for all $d \in \Z_{\geq 1}$.
To first approximation, $N_{M,A}\lll \psi^{c_1(A)-2}\pt\rrr$ can be thought of as the count 
(up to a combinatorial factor) of rational pseudoholomorphic curves in $M$ in homology class $A$ with prescribed $(c_1(A)-1)$-jet at a point. However, compared with the corresponding local tangency $\lll \mathcal{T}^{c_1(A)-2}\pt\rrr$, the descendant counts typically receive extra contributions from boundary strata with ghost components.

Let us introduce one more evenly graded $\Li$ algebra $C_o$ with basis given by formal symbols $\dg_i$ for $i \in \Z_{\geq 1}$.
Put $\mc_{M,A}^o := N_{M,A}\lll \psi^{c_1(A)-2}\pt \rrr \,\dg_{3d-1}$ and
\begin{align*}
\exp_A(\mc_{M}^o) := \sum_{\substack{k \geq 1\\A_1,\dots,A_k \in H_2(M)\\A_1 + \cdots + A_k = A}} \tfrac{1}{k!}\mc^o_{M,A_1} \odot \cdots \odot \mc^o_{M,A_k}.
\end{align*}

\begin{thm}[\cite{SDEP}]\label{thm:punc_desc}
For each $\veca \in \R_{>0}^n$ we have 
\begin{align}\label{eq:stretch_des}
\wh{\aug}_\veca(\exp_A(\mc_{M}^\veca)) = \exp_A(\mc_{M}^o),
\end{align}
where $\aug_\veca: C_\veca \ra C_o$ is the $\Li$ homomorphism defined by
\begin{align}\label{eq:punc_des_comp}
\aug_\veca^k(\orb^\veca_{i_1},\dots,\orb^\veca_{i_k}) = \frac{1}{(\Ga^\veca_{i_1}+\cdots+\Ga^\veca_{i_k})!}
\end{align}
for each $k,i_1,\dots,i_k \in \Z_{\geq 1}$.
\end{thm}
\NI Here $\Ga^\veca_0,\Ga^\veca_1,\Ga^\veca_2,\dots$ is the higher dimensional generalization of the lattice path $\Ga_0^a,\Ga_1^a,\Ga_2^a,\dots$, namely for $\veca$ rationally independent $\Ga_k^\veca$ is the tuple $(i_1,\dots,i_n) \in \Z_{\geq 0}^n$ which minimizes $\max\{a_1i_1,\dots,a_ni_n\}$ subject to $i_1+\cdots+i_n = k$.

At this point the statement of Theorem~\ref{thm:punc_desc} is formally self-contained, but we should give at least some geometric intuition. In essence, \eqref{eq:stretch_des} is the algebraic relation given by neck stretching closed curve stationary descendants in $M$ along the boundary of (a rescaling of) the ellipsoid $E(\veca)$.
Meanwhile, \eqref{eq:punc_des_comp} is the computation of punctured curve stationary descendants in the symplectic completion of $E(\veca)$.
More precisely, for $\veca \in \R_{>0}$ with rationally independent components, we identify $\orb^\veca_k$ with the Reeb orbit of $k$th smallest action (or equivalently Conley--Zehnder index $n-1+2k$) in $\bdy E(\veca)$.
Then $\aug_\veca^k(\orb^\veca_{i_1},\dots,\orb^\veca_{i_k})$ encodes the count of rational pseudoholomorphic curves in $\wh{E}(\veca)$ with positive punctures asymptotic to the Reeb orbits $\orb_{i_1}^\veca,\dots,\orb^\veca_{i_k}$ and carrying the stationary descendant condition $\lll \psi^{c_1(A)-2}\pt\rrr$ (see \cite[\S3.2]{SDEP}).

\sss

Let $\eta_\veca: C_o \ra C_\veca$ denote $\Li$ homomorphism inverse to $\aug_\veca: C_\veca \ra C_o$, whose explicit construction is provided by Proposition~\ref{prop:inverting_L_inf}.
In particular, $\eta_\veca^1$ is the linear inverse of $\aug_\veca^1$, i.e. for $k \in \Z_{\geq 1}$ we have
\begin{align*}
\aug_\veca^1(\orb_k^\veca) = \frac{\dg_k}{(\Ga^\veca_k)!}\;\;\;\;\;\text{and}\;\;\;\;\; \eta_\veca^1(\dg_k) = (\Ga^\veca_k)! \orb_k^\veca.
\end{align*}

As above, let $\pi_1: \ovl{S}C_a \ra C_a$ denote the projection to tensors of word length one.
Applying $\pi_1 \circ \wh{\eta}_\veca$ to both sides of ~\eqref{eq:stretch_des} gives
\begin{align*}
\pi_1(\exp_A(\mc_M^\veca)) = (\pi_1 \circ \wh{\eta}_\veca)(\exp_A(\mc_M^o)),
\end{align*}
i.e. 
\begin{align*}
\mc_{M,A}^\veca &= \sum_{\substack{k \geq 1\\A_1,\dots,A_k \in H_2(M)\\A_1 + \cdots + A_k = A}} \tfrac{1}{k!}\eta_\veca^k(\mc^o_{M,A_1},\dots,\mc^o_{M,A_k})\\
&= \sum_{\substack{k \geq 1\\A_1,\dots,A_k \in H_2(M)\\A_1 + \cdots + A_k = A}} \tfrac{1}{k!}\sum_{T \in \treesor_k} (-1)^{|\Vin(T)|} \eta_\veca^T(\mc^o_{M,A_1},\dots,\mc^o_{M,A_k}).
\end{align*}
Using $\mc_{M,A}^\veca = \wtTcount_{M,A}^\veca \cdot \orb_{c_1(A)-1}^\veca$ and $\mc^o_{M,A} = N_{M,A}\lll \psi^{c_1(A)-2}\pt\rrr \cdot \dg_{c_1(A)-1}$, we can rewrite the above as
\begin{multline}\label{eq:tree_recursion_T_tilde}
\wtTcount_{M,A}^\veca \cdot \orb^\veca_{c_1(A)-1} \\= \sum_{\substack{k \geq 1\\A_1,\dots,A_k \in H_2(M)\\A_1 + \cdots + A_k = A}} \tfrac{1}{k!}\left(\prod_{s=1}^k N_{M,A_s}\lll \psi^{c_1(A_s)-2}\pt\rrr \right)\sum_{T \in \treesor_k} (-1)^{|\Vin(T)|} \eta_\veca^T(\dg_{c_1(A_1)-1},\dots,\dg_{c_1(A_k)-1}).
\end{multline}

We now specialize to the case $M = \CP^2$ and $A = d[L]$, so that
\eqref{eq:tree_recursion_T_tilde} becomes
\begin{align}\label{eq:tree_recursion_T_tilde_CP2}
 \wtTcount_{d}^a \cdot \orb^a_{3d-1} = \sum_{\substack{k \geq 1\\d_1,\dots,d_k \in \Z_{\geq 1}\\d_1 + \cdots + d_k = d}} \tfrac{1}{k!(d_1!)^3\cdots(d_k!)^3}\sum_{T \in \treesor_k} (-1)^{|\Vin(T)|} \eta_a^T(\dg_{3d_1-1},\dots,\dg_{3d_k-1}).
\end{align}
Here similar to above we put $\orb^a_k := \orb^{(1,a)}_k$, $\aug_a := \aug_{(1,a)}$, and so on.

Our goal is to rewrite \eqref{eq:tree_recursion_T_tilde_CP2} as a sum over trees with exactly $d$ leaves.
We will introduce relevant notation as we need it.
\begin{notation}
Given $d_1,\dots,d_k \in \Z_{\geq 1}$ with $d_1+\cdots+d_k = d$, let $\calP(d_1,\dots,d_k)$ denote the set of surjections $h: \{1,\dots,d\} \twoheadrightarrow \{1,\dots,k\}$
such that $|h^{-1}(i)| = d_i$ for $i = 1,\dots,k$.  
\end{notation}
\NI Then we can rewrite \eqref{eq:tree_recursion_T_tilde_CP2} temporarily more redundantly as 
\begin{align*}
\wtTcount_{d}^a \cdot \orb^a_{3d-1} &= \sum_{\substack{k \geq 1\\d_1,\dots,d_k \in \Z_{\geq 1}\\d_1 + \cdots + d_k = d}} \frac{1}{k!(d_1!)^3\cdots(d_k!)^3}{\binom{d}{d_1,\dots,d_k}}^{-1}\sum_{h \in \calP(d_1,\dots,d_k)}\sum_{T \in \treesor_k} (-1)^{|\Vin(T)|} \eta_a^T(\dg_{3d_1-1},\dots,\dg_{3d_k-1}).
\end{align*}
\begin{notation} \hfill
\begin{enumerate}
\item   Put 
  \begin{align*}
   \treesor_k(\{1,\dots,d\}):= \left\{(T,h)\;|\; T \in \treesor_k,\; h: \{1,\dots,d\} \twoheadrightarrow \{1,\dots,k\}\right\},
\end{align*}
i.e. $\treesor_k(\{1,\dots,d\})$ is the set of trees $T \in \treesor_k$ 
equipped with a partition of $\{1,\dots,d\}$ into $k$ parts labeled by the leaves of $T$.

\item Given $(T,h) \in \treesor_k(\{1,\dots,d\})$, put 
\begin{align*}
C_{(T,h)} := \frac{1}{(d_1!)^2\cdots (d_k!)^2}\;\;\;\;\;\text{and}\;\;\;\;\; \aug_a^{(T,h)}(\dg_\bullet) := \aug_a^T(\dg_{3d_1-1},\dots,\dg_{3d_k-1}),
\end{align*}
where $d_i := |h^{-1}(i)|$ for $i = 1,\dots,k$.
\end{enumerate}
\end{notation}
We then have
\begin{align*}
\wtTcount_{d}^a \cdot \orb^a_{3d-1} &= \sum_{k \geq 1}\sum_{(T,h) \in \treesor_k(\{1,\dots,d\})} \tfrac{1}{k!d!}\cdot C_{(T,h)} \cdot   (-1)^{|\Vin(T)|}  \eta_a^{(T,h)}(\dg_\bullet).
\end{align*}

\begin{notation}
Put \begin{align*}
\treesun_k(\{1,\dots,d\}) = \{(T,h)\;|\; T \in \treesun_k,\; h: \{1,\dots,d\} \twoheadrightarrow \Vleaf(T) \},
\end{align*}
i.e. an element of $\treesun_k(\{1,\dots,d\})$ is a tree $T \in \treesun_k$ with $k$ unordered leaves equipped with a partition of $\{1,\dots,d\}$ into $k$ parts labeled by the leaves of $T$. 
Put also 
\begin{align*}
\treesun(\{1,\dots,d\}) := \bigcup\limits_{1 \leq k \leq d} \treesun_k(\{1,\dots,d\}).
\end{align*}

\end{notation}
\NI Note that the symmetric group $\Sigma_k$ acts freely on $\treesor_k(\{1,\dots,d\})$ via $\sigma \cdot (T,h) = (T',\sigma \circ h)$, where $T'$ is given by reordering the leaves of $T$ according to the permutation $\sigma \in \Sigma_k$,
and the quotient is identified with $\treesun_k(\{1,\dots,d\})$.
We then have
\begin{align}\label{eq:just_before_bij}
\wtTcount_{d}^a \cdot \orb^a_{3d-1} &= \sum_{(T,h) \in \treesun(\{1,\dots,d\})}  \tfrac{1}{d!} \cdot C_{(T,h)} \cdot   (-1)^{|\Vin(T)|}  \eta_a^{(T,h)}(\dg_\bullet).
\end{align}

Observe that there is a natural map
\begin{align*}
\zeta: \{(T,\calS)\;|\; T \in \treesor_d, \calS \subset \Vmov(T)\} \ra \treesun(\{1,\dots,d\}),
\end{align*}
defined as follows.
Given $(T,\calS)$, let $T'$ be the tree obtained by removing all edges and vertices strictly above $v$, for each $v \in \calS$.
Note that each $v \in \calS$ becomes a leaf in $T'$.
We define a surjection $h: \{1,\dots,d\} \twoheadrightarrow \Vleaf(T')$ in such a way that
\begin{itemize}
  \item if $v \in \Vleaf(T')$ corresponds to a leaf of $T$, then $h^{-1}(v)$ is the original label of $v$
  \item for $v \in \calS$, $h^{-1}(v)$ is the set of labels of leaves lying above $v$ in $T$.
\end{itemize}
We put $\zeta(T,\calS) = (T',h)$.

In fact, $\zeta$ is a bijection, with inverse map
\begin{align*}
\zeta^{-1}: \treesun(\{1,\dots,d\}) \ra \{(T,\calS)\;|\; T \in \treesor_d, \calS \subset \Vmov(T)\}
\end{align*}
described as follows.
Given $(T,h) \in \treesun(\{1,\dots,d\})$, for each $v \in \Vleaf(T)$ such that $|h^{-1}(v)| \geq 2$ we add $|h^{-1}(v)|$ new leaf vertices, each joined to $v$ by a new leaf edge.
By construction the resulting tree $T'$ comes with a natural bijection $\{1,\dots,d\} \xrightarrow{\sim} \Vleaf(T')$, i.e. $T' \in \treesor_d$.
Also, each $v \in \Vleaf(T)$ satisfying $|h^{-1}(v)| \geq 2$ naturally corresponds to a movable vertex in $T'$, and we denote the set of these by $\calS \subset \Vmov(T')$.
Then we have $\zeta^{-1}(T,h) = (T',\calS)$.

Using this bijection, we rewrite \eqref{eq:just_before_bij} as:

\begin{align}
\wtTcount_{d}^a \cdot \orb^a_{3d-1}  &=  \tfrac{1}{d!}\sum_{T \in \treesor_d}\sum_{\calS \subset \Vmov(T)}  C_{\zeta(T,\calS)} \cdot   (-1)^{|\Vin(\zeta(T,\calS))|}  \eta_a^{\zeta(T,\calS)}(\dg_\bullet).
\end{align}

It will be convenient to extend the notion of leaf number to elements of $\treesun(\{1,\dots,d\})$ as follows.
For $(T,h) \in \treesun(\{1,\dots,d\})$, we define the leaf number $\leaf(v)$ to be the cardinality of $\bigcup\limits_w |h^{-1}(w)|$, where the union is over all leaf vertices $w$ lying above $v$ (including possibly $v$ itself).

\begin{proof}[Proof of Theorem~\ref{thmlet:main_tree}]
According to Proposition~\ref{prop:inverting_L_inf}, $\eta_a^T(\dg_{3d_1-1},\dots,\dg_{3d_k-1})$ is computed as follows.
Recall that $T \in \treesor_k$ has $k$ ordered leaves. For $i = 1,\dots,k$, we label the $i$th leaf edge by $\eta_a^1(\dg_{3d_i-1}) = (\Ga^a_{3d_i-1})!\,\orb^a_{3d_i-1}$.
We then iteratively label each edge of $T$, say with source vertex $v$, by the result of applying $\eta_a^1 \circ \aug_a^j$ to the labels of the incoming edges of $v$ (here the valency of $v$ is $j+1$).

Then for $T \in \treesor_d$ and $\calS \subset \Vmov(T)$, with $\zeta(T,\calS) \in \treesun(\{1,\dots,d\})$ as defined above we have:
\begin{align}\label{eq:eta_a^T}
\eta_a^{\zeta(T,\calS)}(\dg_\bullet) = \prod\limits_{v \in \Vleaf(\zeta(T,\calS))} \left(\Ga^a_{3\leaf(v)-1}\right)! \cdot \dfrac{  \prod\limits_{v \in \Vin(\zeta(T,\calS))} \left(\Ga^a_{3\leaf(v)-1}\right)!}{\prod\limits_{v \in \Vin(\zeta(T,\calS))} \left( \sum\limits_{v' \ra v} \Ga^a_{3\leaf(v')-1}\right)! } \cdot \orb^a_{3d-1}.
\end{align}
Observe that in the case $\calS = \nil$, $\zeta(T,\nil)$ is simply $T$ itself but viewed as an element of $\treesun(\{1,\dots,d\})$, and we have
\begin{align}\label{eq:eta_a^Tnil}
\eta_a^{\zeta(T,\nil)}(\dg_\bullet) = \eta_a^{T}(\underbrace{\dg_2,\dots,\dg_2}_d) = \prod\limits_{v \in \Vleaf(T)} \left(\Ga^a_{3\leaf(v)-1}\right)! \cdot \dfrac{  \prod\limits_{v \in \Vin(T)} \left(\Ga^a_{3\leaf(v)-1}\right)!}{\prod\limits_{v \in \Vin(T)} \left( \sum\limits_{v' \ra v} \Ga^a_{3\leaf(v')-1}\right)! } \cdot \orb^a_{3d-1},
\end{align}
where
\begin{align*}
  \Ga^a_2! = 
\begin{cases}
1 & \text{if}\;\;\;1 < a < 2 \\ 2 & \text{if}\;\;\;a \geq 2.
\end{cases}
\end{align*}  
We have natural a identification $\Vin(T) \approx \Vin(\zeta(T,\calS)) \cup \calS$,
and the leaves of $\zeta(T,\calS)$ correspond to (a) leaves of $T$ not lying above any $v \in \calS$ and (b) elements of $\calS$.
Then, comparing \eqref{eq:eta_a^T} and ~\eqref{eq:eta_a^Tnil}, we have
\begin{align*}
\eta_a^{\zeta(T,\calS)}(\dg_\bullet) &= \eta_a^{\zeta(T,\nil)}(\dg_\bullet) \cdot \frac{\prod\limits_{v \in \calS}\left(\Ga^a_{3\leaf(v)-1} \right)!}{\prod\limits_{v \in \calS} (\Ga^a_2!)^{\leaf(v)}} \cdot \frac{\prod\limits_{v \in \calS}\left(\sum\limits_{v' \ra v} \Ga^a_{3\leaf(v')-1} \right)!}{\prod\limits_{v \in \calS}\left(\Ga^a_{3\leaf(v)-1} \right)! }
\\&= \eta_a^{\zeta(T,\nil)}(\dg_\bullet) \cdot \prod\limits_{v \in \calS} \frac{(\leaf(v)\Ga^a_2)!}{(\Ga^a_2!)^{\leaf(v)}}.
\end{align*}

Note that we have $C_{\zeta(T,\calS)} = \left(\prod\limits_{v \in \calS}\leaf(v)! \right)^{-2}$, 
 and hence
\begin{align*}
\wtTcount_{d}^a \cdot \orb^a_{3d-1}  &= \tfrac{1}{d!}\sum_{T \in \treesor_d}(-1)^{|\Vin(T)|} \sum_{\calS \subset \Vmov(T)}   C_{\zeta(T,\calS)} \cdot   (-1)^{|\calS|}\cdot  \eta_a^{\zeta(T,\calS)}(\dg_\bullet)
\\&= \tfrac{1}{d!}\sum_{T \in \treesor_d}(-1)^{|\Vin(T)|}\cdot \eta_a^{\zeta(T,\nil)}(\dg_\bullet)\sum_{\calS \subset \Vmov(T)}   \left(\prod\limits_{v \in \calS}\leaf(v)! \right)^{-2} \cdot   (-1)^{|\calS|} \cdot \prod\limits_{v \in \calS} \frac{(\leaf(v)\Ga^a_2)!}{(\Ga^a_2!)^{\leaf(v)}}.
\end{align*}
We also have
$$\eta_a^{\zeta(T,\nil)}(\dg_\bullet) =  (\Ga^a_2!)^{d} \cdot \prod\limits_{v \in \Vin(T)}\tfrac{  \left(\Ga^a_{3\leaf(v)-1}\right)!}{\left( \sum\limits_{v' \ra v} \Ga^a_{3\leaf(v')-1}\right)! } \cdot \orb^a_{3d-1},$$
and hence
\begin{align*}
\wtTcount_{d}^a &= \tfrac{1}{d!}(\Ga^a_2!)^{d}\sum_{T \in \treesor_d}(-1)^{|\Vin(T)|} \cdot \prod\limits_{v \in \Vin(T)}\tfrac{ \left(\Ga^a_{3\leaf(v)-1}\right)!}{\left( \sum\limits_{v' \ra v} \Ga^a_{3\leaf(v')-1}\right)! }  \cdot \sum_{\calS \subset \Vmov(T)}     (-1)^{|\calS|} \cdot \prod\limits_{v \in \calS} (\leaf(v)!)^{-2}\frac{(\leaf(v)\Ga^a_2)!}{(\Ga^a_2!)^{\leaf(v)}}
\\&= \tfrac{1}{d!}(\Ga^a_2!)^{d}\sum_{T \in \treesor_d}(-1)^{|\Vin(T)|} \cdot \prod\limits_{v \in \Vin(T)}\tfrac{ \left(\Ga^a_{3\leaf(v)-1}\right)!}{\left( \sum\limits_{v' \ra v} \Ga^a_{3\leaf(v')-1}\right)! }  \cdot \prod\limits_{v \in \Vmov(T)} \left(1 - (\leaf(v)!)^{-2}\frac{(\leaf(v)\Ga^a_2)!}{(\Ga^a_2!)^{\leaf(v)}}\right).
\end{align*}
Finally, we can write this as a sum over $\treesun_d$ by recalling that there is a natural forgetful map $\treesor_d \ra \treesun_d$ and the fiber of any $T \in \treesun_d$ has cardinality $\frac{d!}{|\aut(T)|}$.
\end{proof}

\bibliographystyle{math}
\bibliography{biblio}

\end{document}